%% file: main.tex
\documentclass{birkjour}

\usepackage[utf8]{inputenc}
\usepackage{amsmath,amssymb,amsthm,mathtools}
\usepackage{cleveref}
\usepackage{tikz}
\usepackage{multicol}
\usepackage{enumerate}
\usepackage{url}

\newtheorem{theorem}{Theorem}
\newtheorem{proposition}[theorem]{Proposition}

\DeclareMathOperator{\tr}{tr}
\DeclareMathOperator{\im}{im}

\DeclareMathOperator{\rank}{rank}

\begin{document}

\title[Competitive bimolecular oscillators]{Some minimal bimolecular mass-action\\ systems with limit cycles}
\author{Bal\'azs Boros}
\address{Department of Mathematics, University of Vienna\\
Oskar-Morgenstern-Platz 1, 1090 Vienna, Austria}
\email{balazs.boros@univie.ac.at}
\thanks{BB was supported by the Austrian Science Fund (FWF), project P32532.}

\author{Josef Hofbauer}
\address{Department of Mathematics, University of Vienna\\
Oskar-Morgenstern-Platz 1, 1090 Vienna, Austria}
\email{josef.hofbauer@univie.ac.at}




\begin{abstract}
We discuss three examples of bimolecular mass-action systems with three species, due to Feinberg, Berner, Heinrich, and Wilhelm. Each system has a unique positive equilibrium which is unstable for certain rate constants and then exhibits stable limit cycles, but no chaotic behaviour. For some rate constants in the Feinberg--Berner system, a stable equilibrium, an unstabe limit cycle, and a stable limit cycle coexist. All three networks are minimal in some sense.

By way of homogenising these three examples, we construct bimolecular mass-conserving mass-action systems with four species that admit a stable limit cycle. The homogenised Feinberg--Berner system and the homogenised Wilhelm--Heinrich system admit the coexistence of a stable equilibrium, an unstable limit cycle, and a stable limit cycle.
\end{abstract}

\maketitle

\input{sections/1_intro}
\input{sections/2_prelim}
\input{sections/3_F}
\input{sections/4_WH}
\input{sections/5_W}
\newpage
\bibliographystyle{abbrv}
\bibliography{biblio}

\end{document}

%% file: sections/1_intro.tex
\section{Introduction}

We proved recently that rank-two bimolecular mass-action systems do not admit limit cycles \cite[Section 4]{boros:hofbauer:2022a}. Thus, to construct simple mass-action systems with limit cycles, one could focus either on rank-two trimolecular (or tetramolecular) networks or rank-three bimolecular networks. The former direction is followed in \cite[Section 3]{boros:hofbauer:2022a}, while the latter one is the topic of the present paper. We analyse three known three-species bimolecular mass-action systems that are all minimal in some sense:
\begin{enumerate}[(i)]
\item The underlying reaction network of the \emph{Feinberg--Berner oscillator} (studied in \Cref{sec:F}) is reversible with $5$ complexes and $4$ reversible reactions. As we will discuss, any reversible bimolecular mass-action system with a limit cycle has at least $5$ complexes and at least $4$ reversible reactions.
\item The r.h.s.\ of the differential equation of the \emph{Wilhelm--Heinrich oscillator} (studied in \Cref{sec:WH}) has only one nonlinear term. No linear differential equation admits a limit cycle.
\item The underlying reaction network of the \emph{Wilhelm oscillator} (studied in \Cref{sec:W}) has only $4$ reactions. As we will discuss, any bimolecular mass-action system with a limit cycle has at least $4$ reactions.
\end{enumerate}
Having available the inheritance results summarised in \cite[Theorem 2]{banaji:2022}, it is apparent that finding small networks with limit cycles is a useful tool for studying the capacity for oscillations in more complex networks.

Since physically realistic systems are often mass-conserving, it is also interesting to find small mass-conserving bimolecular networks whose associated mass-action system admits limit cycle oscillation. Following \cite[Section 4]{banaji:boros:hofbauer:2022}, we homogenise each of the above three-species oscillators, and thereby construct three four-species mass-conserving bimolecular mass-action systems that all exhibit limit cycles.

After collecting the necessary preliminaries in \Cref{sec:prelim}, we study the (homogenised) Feinberg--Berner oscillator, the (homogenised) Wilhelm--Hein\-rich oscillator, and the (homogenised) Wilhelm oscillator in \Cref{sec:F,sec:WH,sec:W}, respectively.

%% file: sections/2_prelim.tex
\section{Preliminaries} \label{sec:prelim}

In this section we briefly summarise the necessary background and terminology. For more details on reaction network theory, consult e.g. \cite{yu:craciun:2018} and the references therein. The symbols $\mathbb{R}_+$, $\mathbb{R}_{\geq0}$, and $\mathbb{Z}_{\geq0}$ denote the set of positive reals, nonnegative reals, and nonnegative integers, respectively. 

\subsection{Reaction networks and mass-action systems}

A \emph{reaction network} (or \emph{network} for short) consists of species, complexes, and reactions. For given \emph{species} $\mathsf{X}_1,\ldots,\mathsf{X}_n$, \emph{complexes} are formal linear combinations $a_1\mathsf{X}_1 + \cdots + a_n\mathsf{X}_n$ with nonnegative integers $a_1,\ldots,a_n$. A \emph{reaction} is an ordered pair of a reactant complex and a product complex. With the symbols $V\subseteq\mathbb{Z}^n_{\geq0}$ and $E\subseteq V \times V$ denoting the set of complexes and the set of reactions, respectively, $(V,E)$ is a directed graph. The \emph{reaction vector} associated to a reaction $a_1\mathsf{X}_1 + \cdots + a_n\mathsf{X}_n \to b_1\mathsf{X}_1 + \cdots + b_n\mathsf{X}_n$ is the vector $(b_1-a_1,\ldots,b_n-a_n)^\top$. The \emph{stoichiometric matrix} is denoted by $\Gamma$, its columns are the reaction vectors. By definition, the \emph{rank} of a network is $\rank\Gamma$.

The vector $x(\tau)\in\mathbb{R}^n_{\geq0}$ encodes the concentrations of the $n$ species at time $\tau$. Assuming \emph{mass-action kinetics}, the \emph{rate function} associated to a reaction $a_1\mathsf{X}_1 + \cdots + a_n\mathsf{X}_n \to b_1\mathsf{X}_1 + \cdots + b_n\mathsf{X}_n$ is $\kappa x_1^{a_1} \cdots x_n^{a_n}$, where $\kappa>0$ is called the \emph{rate constant} of the reaction. The rate functions of the individual reactions are collected in the vector $v(x)$. The species concentrations then evolve according to the mass-action differential equation $\dot{x}(\tau) = \Gamma v(x(\tau))$ with state space $\mathbb{R}^n_{\geq0}$. Notice that its solutions are confined to the linear manifolds $(p+\im \Gamma)\cap\mathbb{R}^n_{\geq0}$ with $p\in\mathbb{R}^n_{\geq0}$, termed \emph{stoichiometric classes}. The sets $(p+\im \Gamma)\cap\mathbb{R}^n_+$  with $p\in\mathbb{R}^n_+$ are called \emph{positive stoichiometric classes} and are  also known to be  forward invariant. A reaction network is \emph{mass-conserving} if the stoichiometric classes are bounded (or equivalently, there exists a $d \in \mathbb{R}^n_+$ such that $d^\top\Gamma = 0$).

\subsection{Deficiency}

We now recall some special cases of the classical Deficiency-Zero Theorem \cite[Theorem 3.1]{feinberg:1979} and the classical Deficiency-One Theorem \cite[Theorem 3.2]{feinberg:1979}. The \emph{deficiency} of a network $(V,E)$ is the nonnegative integer $|V|-\ell-\rank\Gamma$, where $\ell$ is the number of connected components of the directed graph $(V,E)$. A directed graph is \emph{strongly connected} if for any two vertices $i$ and $j$ there exists a directed path from $i$ to $j$.
\begin{theorem}\label{thm:dfc0_noLC}
Deficiency-zero mass-action systems do not admit periodic solutions in the positive orthant $\mathbb{R}^n_+$.
\end{theorem}
\begin{theorem}\label{thm:dfc1}
If the deficiency-one network $(V,E)$ is strongly connected and $\rank\Gamma=n$ holds then its associated mass-action system has exactly one positive equilibrium.
\end{theorem}

\subsection{Permanence}

The mass-action differential equation $\dot{x}(\tau) = \Gamma v(x(\tau))$ is said to be \emph{permanent in a positive stoichiometric class} $\mathcal{P}$  if there exists a compact and forward invariant set $K \subseteq \mathcal{P}$ with the property that for each solution $\tau \mapsto x(\tau )$ with $x(0) \in \mathcal{P}$, there exists a $\tau_0 \geq 0$ such that for all $\tau \geq \tau_0$ we have $x(\tau) \in K$. A mass-action system is \emph{permanent} if it is permanent in every positive stoichiometric class. The following theorem was proved in \cite[Theorem 1.3]{gopalkrishnan:miller:shiu:2014}, \cite[Theorem 5.5]{anderson:cappelletti:kim:nguyen:2020}, and \cite[Theorem 4.2]{boros:hofbauer:2020}.
\begin{theorem}\label{thm:permanence1}
If the network $(V,E)$ is strongly connected then its associated mass-action system is permanent.
\end{theorem}

We also use special cases of \cite[Theorems 4.1 and 4.2]{hofbauer:so:1989} to prove permanence. Though the results in \cite{hofbauer:so:1989} are phrased for discrete-time systems, they also hold in the continuous-time case with almost identical proofs. Let $\Delta$ be the simplex $\{x\in \mathbb{R}^n_{\geq0} \colon x_1+\cdots+x_n=c\}$ for some $c>0$, and $F = (F^{\tau})_{\tau \geq 0}$ a continuous semi-flow on $\Delta$ such that the relative interior $\Delta^\circ$ of $\Delta$ is forward invariant. As above, we say that $F$ is permanent if there exists a compact and forward invariant set $K \subseteq \Delta^\circ$ such that for each $x \in \Delta^\circ$  there exists a $\tau_0 \geq 0$ such that for all $\tau \geq \tau_0$ we have $F^{\tau}(x) \in K$. 
Let $M$ be the maximal invariant subset (i.e., the union of all complete orbits) of $\partial \Delta = \Delta \setminus \Delta^\circ$ under $F$. The \emph{stable set} of a compact invariant set $K$ is defined as $W^s(K) = \{ x\in \Delta \colon \omega(x) \subseteq K\}$. We say that a compact invariant set $K$ is \emph{isolated} in $\Delta$ if there is a neighborhood $U$ of $K$ in $\Delta$ such that $K$ is the maximal invariant set in $U$.
\begin{theorem}\label{thm:permanence2a}
If $M$ is isolated in $\Delta$ and $W^s(M)\subseteq \partial \Delta$ then $F$ is permanent.
\end{theorem}
\begin{theorem}\label{thm:permanence2b}
Let $M_1, M_2 \subseteq M$ be invariant subsets of $M$, such that $\alpha(x) \subseteq M_1$  and $\omega(x) \subseteq M_2$ hold for all $x \in M \setminus (M_1 \cup M_2)$. If both $M_1$ and $M_2$ are isolated in $\Delta$ and their stable sets satisfy $W^s(M_i)\subseteq \partial\Delta$ then $F$ is permanent.
\end{theorem}


\subsection{Molecularity}

The \emph{molecularity} of a complex $a_1\mathsf{X}_1 + \cdots + a_n\mathsf{X}_n$ is the nonnegative integer $a_1+\cdots+a_n$. We say that a reaction network is \emph{bimolecular} if every complex in the network has molecularity at most two. The following theorem is proven in \cite[Section 4]{boros:hofbauer:2022a}.
\begin{theorem}\label{thm:ranktwo_bimolec}
Rank-two bimolecular mass-action systems do not admit limit cycles.
\end{theorem}

\subsection{Homogenisation}

We now recall a construction, \emph{homogenisation}, from \cite[Section 4]{banaji:boros:hofbauer:2022}. Start with a mass-action system that admits $r\geq1$ hyperbolic limit cycles and add a new species $\mathsf{X}_{n+1}$ to each reaction in the network in a way that the reactant and the product complex have the same molecularity in the new network (notice that the new network as a directed graph might not be isomorphic to the original one). Then, by \cite[Theorem 1]{banaji:boros:hofbauer:2022}, the new mass-action system admits at least $r$ hyperbolic limit cycles. Limit cycles that are born via Andronov--Hopf bifurcation are hyperbolic (at least for parameters that are close enough to the critical value).

\subsection{Competitive systems}

Finally, we recall some basic facts about competitive systems \cite{hirsch:smith:2005}. With $(a_1,b_1)$, $(a_2,b_2)$, $(a_3,b_3)$ being (possibly unbounded) open intervals in $\mathbb{R}$ let $D=(a_1,b_1)\times(a_2,b_2)\times(a_3,b_3)$. Further, let $f\colon D \to \mathbb{R}^3$ be analytic, and consider the differential equation
\begin{align}\label{eq:fgh}
\dot{x} = f(x).
\end{align}
Let $J(x)$ denote the Jacobian matrix of the r.h.s.\ of \eqref{eq:fgh} at $x\in D$. The system \eqref{eq:fgh} is called \emph{competitive} in $D$ if all the off-diagonal entries of $J$ are nonpositive everywhere in $D$. The system \eqref{eq:fgh} is called \emph{irreducible} in $D$ if $J$ is irreducible everywhere in $D$.

The most important result about three-dimensional competitive systems is the Poincar\'e--Bendixson Theorem, due to Morris Hirsch and Hal Smith, see \cite[Theorem 3.23]{hirsch:smith:2005}. In particular, it excludes chaotic behaviour.
\begin{theorem}
Assume that the differential equation \eqref{eq:fgh} is  competitive in $D$. Then every nonempty compact $\omega$-limit set either contains an equilibrium or is a periodic orbit.
\end{theorem}

The following theorem is due to Hsiu-Rong Zhu and Hal Smith \cite[Theorem 1.2]{zhu:smith:1994}.
\begin{theorem}\label{thm:zhu_smith}
Assume that the differential equation \eqref{eq:fgh} has the following three properties.
\begin{enumerate}[(A)]
    \item It is permanent in $D$. 
    \item It is competitive and irreducible in $D$.
    \item The domain $D$ contains a unique equilibrium $x^*$, the determinant of $J(x^*)$ is negative, and $x^*$ is unstable.
\end{enumerate}
Then there is at least one but no more than finitely many periodic orbits and at least one of them is orbitally asymptotically stable. \end{theorem}
Throughout the paper, whenever we write \emph{stable limit cycle}, we mean it is orbitally asymptotically stable.

%% file: sections/3_F.tex
\section{Feinberg--Berner oscillator} \label{sec:F}

In 1979, Martin Feinberg introduced the mass-action system
\begin{align}\label{eq:F_original_ode}
\begin{aligned}
\input{tikz/F_original}
\end{aligned}
\end{align}
in \cite[Example 3.D.3]{feinberg:1979} as a perturbation of the Edelstein mass-action system \cite{edelstein:1970}. The Edelstein network is obtained from \eqref{eq:F_original_ode} by deleting the reversible reaction $\mathsf{X} \rightleftarrows \mathsf{Z}$ (which amounts to setting $\kappa_3=\kappa_4=0$ in the differential equation). By the Deficiency-One Theorem (\Cref{thm:dfc1} in \Cref{sec:prelim}), the differential equation \eqref{eq:F_original_ode} has exactly one positive equilibrium for all choices of rate constants. This holds despite the fact that the Edelstein mass-action system does admit multiple positive equilibria (for some rate constants in some stoichiometric class).

The additional purpose of Feinberg for presenting example \eqref{eq:F_original_ode} is to demonstrate that not all of the conclusions of the Deficiency-Zero Theorem hold true for networks under the scope of the Deficiency-One Theorem. In short, though the existence and the uniqueness of the positive equilibrium hold true, its asymptotic stability in general does not follow.
According to \cite[Remark 3.5]{feinberg:1979}, Paul Berner identified rate constants for which the unique positive equilibrium of \eqref{eq:F_original_ode} is unstable and numerical solutions suggest the existence of a stable limit cycle. The same remark is made in \cite[page 108]{feinberg:1980} and \cite[page 66]{feinberg:1981}. The mass-action system \eqref{eq:F_original_ode} appears further in \cite[Remark 6.2.B]{feinberg:1987} and \cite[(4.12)]{feinberg:1995a}. To our knowledge, the only place, where specific rate constants are given for which the unique positive equilibrium is claimed to be unstable is \cite[(8.5)]{feinberg:2019}. However, for those rate constants all the three eigenvalues of the Jacobian matrix at the unique positive equilibrium are real and negative, and thus, the equilibrium is linearly stable. After we pointed this out, Feinberg noticed (and kindly let us know via private communication) that the values of $\kappa_3$ and $\kappa_4$ are accidentally swapped in \cite[(8.5)]{feinberg:2019} due to a typographical error. Indeed, setting the values of $\kappa_3$ and $\kappa_4$ correctly, one finds that the equilibrium is unstable, as claimed in \cite[Section 8.3]{feinberg:2019}.

Below we prove that there exist rate constants such that the unique positive equilibrium of the mass-action system \eqref{eq:F_original_ode} is unstable and is surrounded by a stable limit cycle. Furthermore, we also prove that there exist rate constants such that the positive equilibrium is asymptotically stable and is surrounded by two limit cycles (an unstable and a stable one).

Before we turn to \Cref{thm:F_original}, the analysis of the mass-action system \eqref{eq:F_original_ode}, we discuss in what sense the network is minimal. A reaction network is called \emph{reversible} if along with any reaction $i \to j$, the reverse reaction $j \to i$ is also present in the network. The network \eqref{eq:F_original_ode} is reversible and it has $5$ complexes and $4$ reversible reactions. The following proposition states that neither the number of complexes nor the number of reactions could be less for a reversible bimolecular network whose associated mass-action system admits a limit cycle.

\begin{proposition}\label{prop:FB_minimal}
Any reversible bimolecular network whose associated mass-action system admits a limit cycle has
\begin{enumerate}[(i)]
\item at least $5$ complexes and
\item at least $4$ reversible reactions.
\end{enumerate}
\end{proposition}
\begin{proof}
By \Cref{thm:ranktwo_bimolec}, the rank of a bimolecular mass-action system with a limit cycle is at least $3$. For the rank to be at least $3$, it must have at least $4$ complexes. However, with only $4$ complexes, the deficiency is $0$, and that precludes the possibility of a limit cycle by \Cref{thm:dfc0_noLC}. This concludes the proof of (i). If the network has only $3$ reversible reactions, and the rank of the network is $3$ then it is not hard to see that, again, the deficiency is zero, which, again, precludes the possibility of a limit cycle. This concludes the proof of (ii).
\end{proof}

\begin{theorem}\label{thm:F_original}
For the mass-action system \eqref{eq:F_original_ode} the following statements hold.
\begin{enumerate}[(i)]
\item The system is permanent.
\item The system is competitive.
\item If the unique positive equilibrium is unstable then there exists a stable limit cycle and there are only finitely many periodic orbits.
\item The unique positive equilibrium $(x^*,y^*,z^*)$ equals $(1,1,1)$ if and only if
\begin{align}\label{eq:111}
\kappa_3 = \kappa_4 \text{ and } \kappa_1 -\kappa_2 = \kappa_5 - \kappa_6 = \kappa_7-\kappa_8. 
\end{align}
\item If \eqref{eq:111} and $\kappa_2+2\kappa_4 \geq \kappa_1$ hold then $(1,1,1)$ is linearly stable.
\item Set $\kappa_1 = 1$, $\kappa_2=\frac15$, $\kappa_4=\frac15$ and eliminate $\kappa_3$, $\kappa_5$, $\kappa_7$ by \eqref{eq:111} (thus, the only free parameters are $\kappa_6>0$ and $\kappa_8>0$).
\begin{enumerate}[(a)]
\item Then the positive equilibrium $(1,1,1)$ is asymptotically stable on $\mathcal{H}_s$, undergoes an Andronov--Hopf bifurcation at $\mathcal{H}_c$, and is unstable on $\mathcal{H}_u$, where
\begin{align*}
\mathcal{H}_s &= \{(\kappa_6,\kappa_8)\in\mathbb{R}^2_+ \colon h(\kappa_6,\kappa_8)>0\},\\
\mathcal{H}_c &= \{(\kappa_6,\kappa_8)\in\mathbb{R}^2_+ \colon h(\kappa_6,\kappa_8)=0\},\\
\mathcal{H}_u &= \{(\kappa_6,\kappa_8)\in\mathbb{R}^2_+ \colon h(\kappa_6,\kappa_8)<0\}
\end{align*}
with $h(\kappa_6,\kappa_8)= 50 \kappa_6^2 \kappa_8 + 100 \kappa_6 \kappa_8^2 + 55 \kappa_6^2 + 260 \kappa_6 \kappa_8 + 50 \kappa_8^2  + 128 \kappa_6 + 40 \kappa_8  - 26$.
\item Along the curve $\mathcal{H}_c$, the first focal value is negative for $\kappa_6$ close to $0$, while it is positive for $\kappa_8$ close to $0$. Thus, the Andronov--Hopf bifurcation could be supercritical, subcritical, and degenerate.
\end{enumerate}
\item There exist rate constants such that the unique positive equilibrium is unstable and is surrounded by a stable limit cycle.
\item There exist rate constants such that the unique positive equilibrium is asymptotically stable and is surrounded by two limit cycles (an unstable and a stable one).
\end{enumerate}
\end{theorem}
\begin{proof}
Since the network is strongly connected, it is permanent by \Cref{thm:permanence1}, thus, (i) follows.

With $\widetilde{z}=-z$, the differential equation \eqref{eq:F_original_ode} in the new coordinates $(x,y,\widetilde{z})$ take the form
\begin{align*}
\begin{split}
\dot{x} &= - \kappa_1 x y - \kappa_2 \widetilde{z} - \kappa_3 \widetilde{z} - \kappa_4 x + \kappa_5 x - \kappa_6 x^2, \\
\dot{y} &= - \kappa_1 x y - \kappa_2 \widetilde{z} - \kappa_7 \widetilde{z} - \kappa_8 y, \\
\dot{\widetilde{z}} &= -\kappa_1 x y - \kappa_2 \widetilde{z} - \kappa_3 \widetilde{z} - \kappa_4 x - \kappa_7 \widetilde{z} - \kappa_8 y
\end{split}
\end{align*}
with state space $\mathbb{R}_+ \times \mathbb{R}_+ \times \mathbb{R}_-$. The Jacobian matrix of the r.h.s.\ equals
\begin{align*}
\begin{bmatrix}
-\kappa_1 y - \kappa_4 + \kappa_5 - 2\kappa_6 x & -\kappa_1 x & -\kappa_2-\kappa_3 \\
-\kappa_1 y & -\kappa_1 x - \kappa_8 & -\kappa_2-\kappa_7 \\
-\kappa_1 y - \kappa_4 & -\kappa_1 x - \kappa_8 & -\kappa_2 - \kappa_3 - \kappa_7 \\
\end{bmatrix},
\end{align*}
a matrix whose off-diagonal entries are negative on $\mathbb{R}_+ \times \mathbb{R}_+ \times \mathbb{R}_-$, the system is thus competitive, proving (ii).

Statement (iii) follows by the application of \Cref{thm:zhu_smith}. The determinant of the Jacobian matrix at the unique positive equilibrium is negative by \cite[Theorem 5(iii)(b)]{boros:hofbauer:2022a}.

Statement (iv) follows by a short calculation.

Next, we prove (v). We eliminate $\kappa_3$, $\kappa_5$, $\kappa_7$ by \eqref{eq:111}. Thus, the remaining parameters are $\kappa_1$, $\kappa_2$, $\kappa_4$, $\kappa_6$, $\kappa_8$ with all of them being positive and additionally both of $\kappa_6 + \kappa_1   - \kappa_2$ and $ \kappa_8  + \kappa_1 - \kappa_2$ are also positive. The Jacobian matrix at the unique positive equilibrium $(1,1,1)$ equals
\begin{align*}
\begin{bmatrix}
-\kappa_2-\kappa_4-\kappa_6 & -\kappa_1 & \kappa_2 + \kappa_4 \\
-\kappa_1 & -\kappa_1-\kappa_8 & \kappa_1+\kappa_8 \\
\kappa_1+\kappa_4 & \kappa_1+\kappa_8 & -\kappa_1-\kappa_4-\kappa_8
\end{bmatrix}.
\end{align*}
The characteristic polynomial of this matrix equals $\lambda^3 + a_2 \lambda^2 + a_1 \lambda + a_0$, where
\begin{align*}
a_2 &= 2\kappa_1 + \kappa_2 + 2\kappa_4 + \kappa_6 + 2\kappa_8, \\
a_1 &= \kappa_1 (\kappa_2+2\kappa_4-\kappa_1) + 2\kappa_2 \kappa_8 + \kappa_4 \kappa_6 + 3 \kappa_4 \kappa_8 + 2\kappa_1 \kappa_6  +2\kappa_6 \kappa_8  , \\
a_0 &= \kappa_4(\kappa_1\kappa_6+\kappa_1\kappa_8+\kappa_6\kappa_8).
\end{align*}
By the Routh--Hurwitz criterion, all the three eigenvalues have negative real part if and only if $a_0$, $a_2$, $a_2 a_1 - a_0$ are all positive. Clearly, $a_0$ and $a_2$ are positive. Under $\kappa_2 + 2\kappa_4 - \kappa_1 \geq 0$, the product $a_2 a_1$ is positive as well. The positivity of $a_2 a_1 - a_0$ then follows, because the negative terms $-\kappa_1\kappa_4\kappa_6$, $-\kappa_1\kappa_4\kappa_8$, $-\kappa_4\kappa_6\kappa_8$ coming from $-a_0$ are compensated by the positive terms $6\kappa_1\kappa_4\kappa_6$, $6\kappa_1\kappa_4\kappa_8$, $9\kappa_4\kappa_6\kappa_8$ coming from $a_2 a_1$. This concludes the proof of (v).

To prove (vi), notice that in the special case $\kappa_1= 1$, $\kappa_2=\frac15$, $\kappa_4=\frac15$, the expression $a_2a_1-a_0$ simplifies to $h(\kappa_6,\kappa_8)$ (up to a positive multiplier). Part (a) in (vi) then follows immediately by the Routh--Hurwitz criterion. For proving part (b), one computes the first focal value, $L_1$, along the curve $\mathcal{H}_c$. By applying formula \cite[(5.39)]{kuznetsov:2004} 
we find that $L_1<0$ when $\kappa_6$ is close to $0$, while $L_1>0$ when $\kappa_8$ is close to $0$ (for the calculations, see the respective Mathematica Notebook in the GitHub repository \cite{boros:2022}). We have depicted in the left panel of \Cref{fig:F} the sign of $h(\kappa_6,\kappa_8)$ and the sign of the first focal value. This concludes the proof of (vi).

To prove (vii), it suffices to note that a supercritical Andronov--Hopf bifurcation is possible by (vi).

\begin{figure}[ht]
    \centering
    \begin{tabular}{cc}
    \includegraphics[scale=0.25]{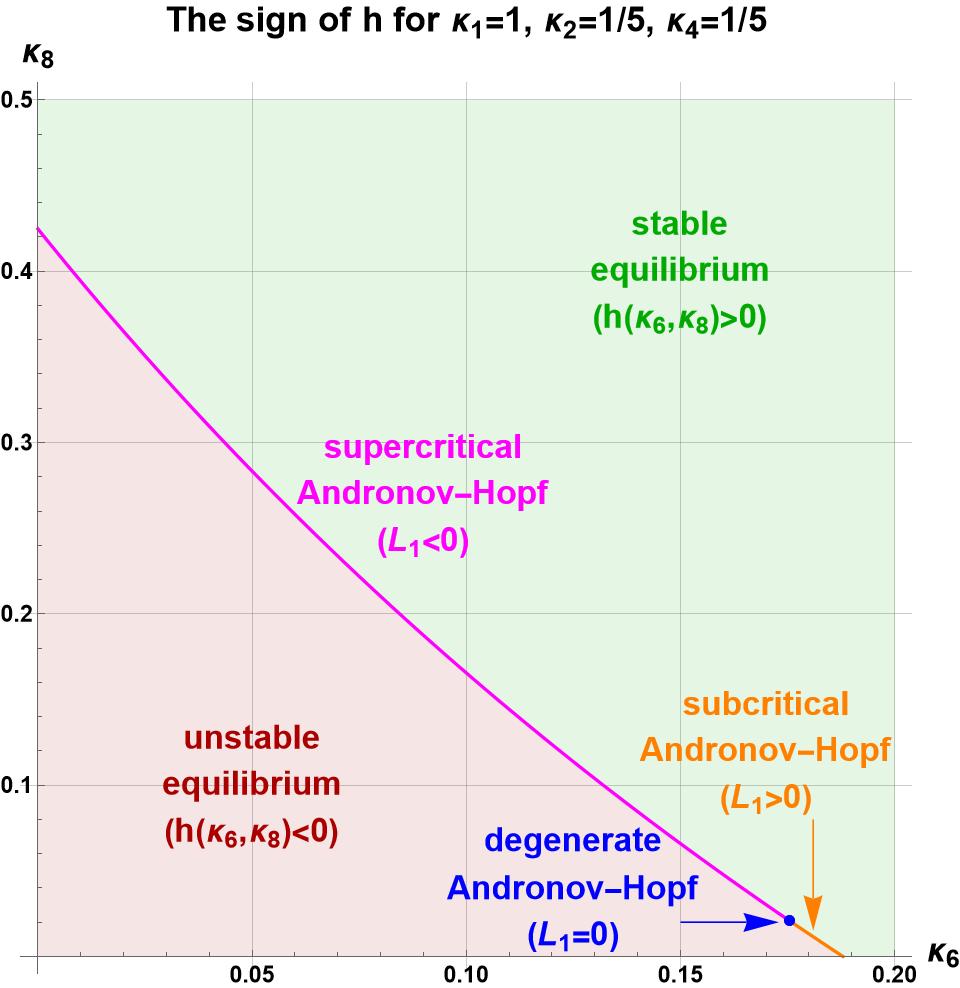} &
    \includegraphics[scale=0.19,trim=1.8cm 0 0 0]{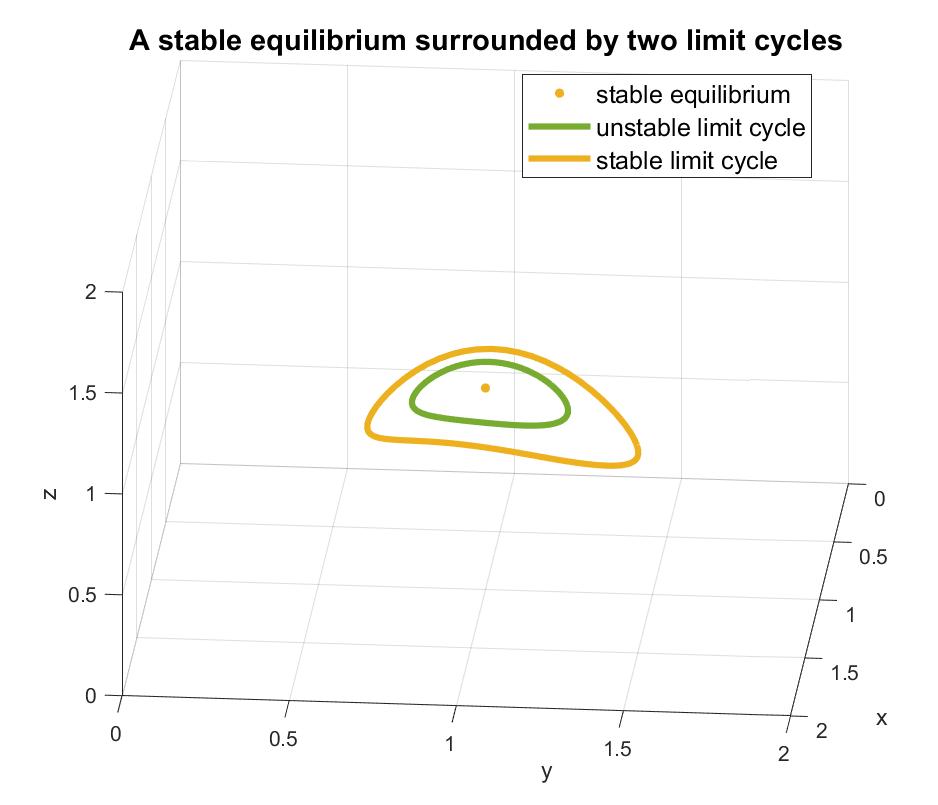}
    \end{tabular}
    \caption{Take the mass-action system \eqref{eq:F_original_ode} with $\kappa_1 = 1$, $\kappa_2 = \frac15$, $\kappa_3 = \frac15$, $\kappa_4=\frac15$, $\kappa_5 = \kappa_6 + \frac{4}{5}$, $\kappa_7 = \kappa_8 + \frac{4}{5}$. Left panel: the bifurcation diagram with $\kappa_6$ and $\kappa_8$ kept as parameters. Right panel: the two limit cycles for $\kappa_6 = 0.187$ and $\kappa_8 = 0.0052$ (notice that these parameters fall slightly above the orange curve of the left panel); the figure was created using MATCONT \cite{dhooge:govaerts:kuznetsov:2003}.}
    \label{fig:F}
\end{figure}

To prove (viii), fix rate constants such that the unique positive equilibrium has a negative real eigenvalue and a pair of purely imaginary eigenvalues, and the first focal value is positive. Such rate constants exist by (vi), see the orange curve in the left panel of \Cref{fig:F}. Since the equilibrium is unstable, there exists a stable limit cycle by (iii). Now perturb the rate constants slightly and create an unstable limit cycle on the center manifold via a subcritical Andronov--Hopf bifurcation. The equilibrium then becomes asymptotically stable. We illustrated the two limit cycles in the right panel of \Cref{fig:F}. This concludes the proof of (viii).

We now show an alternative way to prove (viii) that avoids the application of (iii) and furthermore guarantees that the two limit cycles are hyperbolic. Fix rate constants such that the unique positive equilibrium has a negative real eigenvalue and a pair of purely imaginary eigenvalues, the first focal value vanishes, and the second focal value is negative. One confirms that along the curve $\mathcal{H}_c$ in (vi) there is one such point (shown in blue on the left panel in \Cref{fig:F}). The calculation of the second focal value can be performed numerically by following the method described in \cite[Sections 8.7.1 and 8.7.3]{kuznetsov:2004}, see the respective Mathematica Notebook in the GitHub repository \cite{boros:2022} for the details.
Then perturb the rate constants $\kappa_6$ and $\kappa_8$ slightly along $\mathcal{H}_c$ to make the first focal value positive. Then a stable limit cycle is born via a degenerate Andronov--Hopf bifurcation (or Bautin bifurcation) and the equilibrium becomes unstable. Finally, perturb the rate constants $\kappa_6$ and $\kappa_8$ away from $\mathcal{H}_c$ to $\mathcal{H}_s$. With this, an unstable limit cycle is born via a subcritical Andronov--Hopf bifurcation and the equilibrium becomes asymptotically stable.
\end{proof}

We remark that the choice $\kappa_1=1$, $\kappa_2=\frac15$, $\kappa_4=\frac15$ in part (vi) of \Cref{thm:F_original} does not play a crucial role in the sense that for any $\kappa_1$, $\kappa_2$, $\kappa_4$ with $\kappa_2 + 2\kappa_4<\kappa_1$ we have the same qualitative picture: there is a curve in the positive quadrant of the $(\kappa_6,\kappa_8)$-plane that connects the two axes, and where the unique positive equilibrium undergoes an Andronov--Hopf bifurcation, which is supercritical close to the $\kappa_8$-axis and subcritical close to the $\kappa_6$-axis.

\medskip

Before we turn to the homogenisation of the Feinberg--Berner oscillator \eqref{eq:F_original_ode}, we briefly discuss detailed balance and complex balance \cite{horn:1972,yu:craciun:2018}. For a reversible mass-action system, a positive vector $x$ is \emph{detailed balanced} if for each reversible reaction the rate of the forward and the backward reaction at $x$ equals. For a mass-action system, a positive vector $x$ is \emph{complex balanced} if at every complex $i$ the total rate at $x$ over the reactions whose reactant complex is $i$ equals the total rate at $x$ over the reactions whose product complex is $i$. Clearly, a detailed balanced vector is also complex balanced. Further, a detailed balanced or a complex balanced vector is an equilibrium. Finally, if a mass-action system admits a detailed balanced (respectively, complex balanced) equilibrium then all equilibria are detailed balanced (respectively, complex balanced). The following proposition tells us about detailed balance and complex balance in the mass-action system \eqref{eq:F_original_ode}.

\begin{proposition}\label{prop:F_CB}
For the mass-action system \eqref{eq:F_original_ode} the following are equivalent.
\begin{enumerate}[(a)]
    \item The system is complex balanced.
    \item The system is detailed balanced.
    \item The relation $\frac{\kappa_1}{\kappa_2} \frac{\kappa_5}{\kappa_6} \frac{\kappa_7}{\kappa_8}=1$ holds.
\end{enumerate}
\end{proposition}
\begin{proof}
Note that if we replace each reversible reaction with an undirected edge in \eqref{eq:F_original_ode} then we get a tree. For such reversible networks, complex balance and detailed balance are equivalent. By definition, the mass-action system \eqref{eq:F_original_ode} is detailed balanced if there exists $(x,y,z)\in\mathbb{R}^3_+$ such that
\begin{align} \label{eq:F_CB}
\begin{split}
\kappa_1 x y &= \kappa_2 z, \\
\kappa_3 z &= \kappa_4 x, \\
\kappa_5 x &= \kappa_6 x^2, \\
\kappa_7 z &= \kappa_8 y. \\
\end{split}
\end{align}
Solving the last three equations in \eqref{eq:F_CB} for $x$, $y$, $z$ yields $x=\frac{\kappa_5}{\kappa_6}$, $y = \frac{\kappa_4}{\kappa_3}\frac{\kappa_5}{\kappa_6}$, $z = \frac{\kappa_4}{\kappa_3}\frac{\kappa_5}{\kappa_6}\frac{\kappa_7}{\kappa_8}$. It is compatible with the first equation in \eqref{eq:F_CB} if and only if $\frac{\kappa_1}{\kappa_2} \frac{\kappa_5}{\kappa_6} \frac{\kappa_7}{\kappa_8}=1$ holds. This concludes the proof.
\end{proof}


We now homogenise the network in \eqref{eq:F_original_ode} and obtain the mass-conserving mass-action system
\begin{align} \label{eq:F_homog_ode}
\begin{aligned}
\input{tikz/F_partial_homog}
\end{aligned}
\end{align}
Since the network is reversible, there exists a positive equilibrium in every stoichiometric class $x+y+z+w=c$ with $c>0$ \cite[Theorem 1]{boros:2019}. The uniqueness of the positive equilibrium in each of these stoichiometric classes follows by application of the Deficiency-One Algorithm \cite{feinberg:1988,feinberg:1995b}. In fact, since it is possible to parametrise the set of positive equilibria, one sees directly the existence and uniqueness. In the following proposition we collected a few properties of the mass-action system \eqref{eq:F_homog_ode}.

\begin{theorem}
For the mass-action system \eqref{eq:F_homog_ode} the following statements hold.
\begin{enumerate}[(i)]
\item The set of positive equilibria is $\{(x^*(t),y^*(t),z^*(t),w^*(t))\colon t>0\}$, where
\begin{align*}
x^*(t) &= t,\\
y^*(t) &= \frac{\kappa_2 \kappa_3 \kappa_6 t + \kappa_7(\kappa_2 \kappa_4 + \kappa_3 \kappa_5)}{\kappa_1 \kappa_3 \kappa_5 t + \kappa_8(\kappa_2 \kappa_4 + \kappa_3 \kappa_5)}\frac{\kappa_4}{\kappa_3}t,\\
z^*(t) &= \frac{\kappa_4}{\kappa_3}t, \\
w^*(t) &= \frac{\kappa_1 \kappa_3 \kappa_6 t + (\kappa_1 \kappa_4 \kappa_7 + \kappa_3 \kappa_6 \kappa_8)}{\kappa_1 \kappa_3 \kappa_5 t + \kappa_8(\kappa_2 \kappa_4 + \kappa_3 \kappa_5)}t.
\end{align*}
\item There exists a unique positive equilibrium in every stoichiometric class $x+y+z+w=c$ with $c>0$.
\newpage
\item The following are equivalent.
\begin{enumerate}[(a)]
    \item The system is complex balanced.
    \item The system is detailed balanced.
    \item The relation $\frac{\kappa_1}{\kappa_2} \frac{\kappa_5}{\kappa_6} \frac{\kappa_7}{\kappa_8}=1$ holds.
    \item The set of positive equilibria is $\{(a t^\alpha,b t^\beta, c t^\gamma, d t^\delta)\colon t>0\}$ with some $a$, $b$, $c$, $d$ positive numbers and $\alpha$, $\beta$, $\gamma$, $\delta$ real numbers.
    \item The formula $y^*(t)$ in (i) above is linear in $t$.
    \item The formula $w^*(t)$ in (i) above is linear in $t$.
\end{enumerate}
\item There exist rate constants and a stoichiometric class with an asymptotically stable equilibrium and two limit cycles (an unstable and a stable one).
\item The system is permanent.
\end{enumerate}
\end{theorem}
\begin{proof}
One proves (i) by a direct calculation.

Since for any $a$, $b$, $c$, $d$ positive numbers the function $t \mapsto \frac{at^2+bt}{ct+d}$ is strictly increasing on $(0,\infty)$, statement (ii) follows immediately from the parametrisation in (i).

We now prove (iii). The equivalence of (a), (b), (c) can be shown in the same way as in \Cref{prop:F_CB}. The equivalence of (c), (d), (e), (f) follows by a direct computation, using the parametrisation in (i).

Statement (iv) is an immediate consequence of part (viii) in \Cref{thm:F_original} and \cite[Theorem 1]{banaji:boros:hofbauer:2022}.

To prove (v), we use \Cref{thm:permanence2a}. For any fixed $c>0$ consider the mass-action system \eqref{eq:F_homog_ode} in the tetrahedron
\begin{align*}
\Delta = \{(x,y,z,w)\in\mathbb{R}^4_{\geq0} \colon x+y+z+w=c\}.
\end{align*}
The only equilibrium on the boundary of $\Delta$ is $E_\mathsf{W}=(0,0,0,c)$. If we eliminate $w$ by $w = c-x-y-z$, the $3 \times 3$ Jacobian matrix at the equilibrium $E_\mathsf{W}$ is given by
\begin{align*}
J = \begin{bmatrix}
- \kappa_4 + c \kappa_5 & 0  & c \kappa_2 + \kappa_3 \\
0  &  - \kappa_8 & c\kappa_2 + \kappa_7 \\
 \kappa_4 &   \kappa_8 & -c \kappa_2 - \kappa_3 - \kappa_7
\end{bmatrix}.
\end{align*}
The determinant of this matrix is $c \kappa_8(\kappa_2\kappa_4 + \kappa_3\kappa_5) > 0$. Also, $J$ is quasipositive, i.e., its off-diagonal terms are nonnegative. Hence, the stability modulus of $J$ (i.e., its largest (or rightmost) eigenvalue) is positive. This shows that the equilibrium $E_\mathsf{W}$ is not saturated (see \cite{hofbauer:1990}), and therefore no solution from $\Delta^\circ = \Delta \cap \mathbb{R}^4_+$ converges to $E_\mathsf{W}$, and $\{E_\mathsf{W}\}$ is an isolated invariant set in $\Delta$. Since $\{E_\mathsf{W}\}$ is the maximal invariant set within the boundary of $\Delta$, \Cref{thm:permanence2a} with $M=\{E_\mathsf{W}\}$ shows permanence. 
\end{proof}

%% file: tikz/F_original.tex
\begin{tikzpicture}[scale=1.7]

\node (P1) at (1,0)  {$\mathsf{Y}$};
\node (P2) at (0,0)  {$\mathsf{Z}$};
\node (P3) at (0,1)  {$\mathsf{X}$};
\node (P4) at (1,1)  {$2\mathsf{X}$};
\node (P5) at (-1,0) {$\mathsf{X} + \mathsf{Y}$};

\draw[arrows={-stealth},transform canvas={yshift=-2pt}] (P1) to node[below] {$\kappa_8$} (P2);
\draw[arrows={-stealth},transform canvas={yshift=2pt}] (P2) to node[above] {$\kappa_7$} (P1);
\draw[arrows={-stealth},transform canvas={xshift=2pt}] (P2) to node[right] {$\kappa_3$} (P3);
\draw[arrows={-stealth},transform canvas={xshift=-2pt}] (P3) to node[left] {$\kappa_4$} (P2);
\draw[arrows={-stealth},transform canvas={yshift=2pt}] (P3) to node[above] {$\kappa_5$} (P4);
\draw[arrows={-stealth},transform canvas={yshift=-2pt}] (P4) to node[below] {$\kappa_6$} (P3);
\draw[arrows={-stealth},transform canvas={yshift=-2pt}] (P2) to node[below] {$\kappa_2$} (P5);
\draw[arrows={-stealth},transform canvas={yshift=2pt}] (P5) to node[above] {$\kappa_1$} (P2);

\node at (3.4,1/2){$\begin{aligned}
\dot{x} &= - \kappa_1 x y + \kappa_2 z + \kappa_3 z - \kappa_4 x + \kappa_5 x - \kappa_6 x^2, \\
\dot{y} &= - \kappa_1 x y + \kappa_2 z + \kappa_7 z - \kappa_8 y, \\
\dot{z} &= \kappa_1 x y - \kappa_2 z - \kappa_3 z + \kappa_4 x - \kappa_7 z + \kappa_8 y
\end{aligned}$};

\end{tikzpicture}

%% file: tikz/F_partial_homog.tex
\begin{tikzpicture}[scale=1.7]

\node (P1) at (1/2,4/3) {$\mathsf{X} + \mathsf{Y}$};
\node (P2) at (3/2,4/3) {$\mathsf{Z} + \mathsf{W}$};

\node (P3) at (1/2,2/3) {$\mathsf{X} + \mathsf{W}$};
\node (P4) at (3/2,2/3) {$2\mathsf{X}$};

\node (P5) at (0,0) {$\mathsf{X}$};
\node (P6) at (1,0) {$\mathsf{Z}$};
\node (P7) at (2,0) {$\mathsf{Y}$};

\draw[arrows={-stealth},transform canvas={yshift=2pt}] (P1) to node[above] {$\kappa_1$} (P2);
\draw[arrows={-stealth},transform canvas={yshift=-2pt}] (P2) to node[below] {$\kappa_2$} (P1);

\draw[arrows={-stealth},transform canvas={yshift=2pt}] (P3) to node[above] {$\kappa_5$} (P4);
\draw[arrows={-stealth},transform canvas={yshift=-2pt}] (P4) to node[below] {$\kappa_6$} (P3);

\draw[arrows={-stealth},transform canvas={yshift=2pt}] (P5) to node[above] {$\kappa_4$} (P6);
\draw[arrows={-stealth},transform canvas={yshift=-2pt}] (P6) to node[below] {$\kappa_3$} (P5);

\draw[arrows={-stealth},transform canvas={yshift=2pt}] (P6) to node[above] {$\kappa_7$} (P7);
\draw[arrows={-stealth},transform canvas={yshift=-2pt}] (P7) to node[below] {$\kappa_8$} (P6);

\node at (4.5,2/3)
{$\begin{aligned}
\dot{x} &= - \kappa_1 x y + \kappa_2 z w + \kappa_3 z - \kappa_4 x + \kappa_5 x w - \kappa_6 x^2, \\
\dot{y} &= - \kappa_1 x y + \kappa_2 z w + \kappa_7 z - \kappa_8 y, \\
\dot{z} &= \kappa_1 x y - \kappa_2 z w - \kappa_3 z + \kappa_4 x - \kappa_7 z + \kappa_8 y, \\
\dot{w} &= \kappa_1 x y - \kappa_2 z w - \kappa_5 x w + \kappa_6 x^2.
\end{aligned}$};

\end{tikzpicture}

%% file: sections/4_WH.tex
\newpage
\section{Wilhelm--Heinrich oscillator} \label{sec:WH}

In 1995, Thomas Wilhelm and Reinhart Heinrich presented the smallest bimolecular chemical reaction system with an Andronov--Hopf bifurcation \cite{wilhelm:heinrich:1995}. Namely, they introduced the mass-action system
\begin{align}\label{eq:WH_original_ode}
\begin{aligned}
\input{tikz/WH_original}
\end{aligned}
\end{align}
It is minimal in the sense that it has the smallest number of terms on the r.h.s.\ of the differential equation (two monomials in each line; it is a $3$-dimensional S-system) and there is only one quadratic term. In 1996, the authors studied the Andronov--Hopf bifurcation \cite{wilhelm:heinrich:1996} and in 2012, Hal Smith analysed the system further \cite{smith:2012}. We summarise their findings in the following theorem.

\begin{theorem}\label{thm:WH_original}
For the mass-action system \eqref{eq:WH_original_ode} the following statements hold.
\begin{enumerate}[(i)]
\item For $\kappa_1-\kappa_4>0$ there exists a unique positive equilibrium, it is given by
\begin{align*}
(x^*,y^*,z^*) = \frac{\kappa_1-\kappa_4}{\kappa_2}\left(\frac{\kappa_3}{\kappa_4}, 1, \frac{\kappa_3}{\kappa_5}\right).
\end{align*}
\item The system is competitive.
\item All forward trajectories are bounded, and for $\kappa_1-\kappa_4>0$ the system is permanent.
\item The unique positive equilibrium undergoes a supercritical Andronov--Hopf bifurcation at $\kappa_1 -\kappa_4 = \kappa_3 + \kappa_5$.
\item The unique positive equilibrium is globally asymptotically stable  for $0 < \kappa_1 -\kappa_4 \leq \kappa_3 + \kappa_5$.
\item For $\kappa_1 -\kappa_4 > \kappa_3 + \kappa_5$, the unique positive equilibrium is unstable, it is surrounded by a stable limit cycle, and there are only finitely many periodic orbits.
\end{enumerate}
\end{theorem}

As remarked by Hal Smith, the main open problem about the dynamics of the mass-action system \eqref{eq:WH_original_ode} is the number of limit cycles in case $\kappa_1 -\kappa_4 > \kappa_3 + \kappa_5$. He says, numerical simulations suggest there is only one.

We now homogenise the network in \eqref{eq:WH_original_ode} and obtain the mass-conserving mass-action system
\begin{align} \label{eq:WH_homog_ode}
\begin{aligned}
\input{tikz/WH_partial_homog}
\end{aligned}
\end{align}
By \cite[Theorem 1]{banaji:boros:hofbauer:2022}, the mass-action system \eqref{eq:WH_homog_ode} admits a stable limit cycle, i.e., for some rate constants and in some stoichiometric class there exists a stable limit cycle. In fact, as per the following theorem, for some rate constants in some stoichiometric class the unique positive equilibrium is asymptotically stable and is surrounded by two small limit cycles (an unstable and a stable one). Note that this behaviour is ruled out for the Wilhelm--Heinrich oscillator \eqref{eq:WH_original_ode}.

We now analyse the homogenised Wilhelm--Heinrich oscillator \eqref{eq:WH_homog_ode}. In particular, we show that the Andronov--Hopf bifurcation could be supercritical, subcritical, or degenerate. W.l.o.g.\ we assume $\kappa_1=1$ and for better readability we introduce $p=\kappa_2$, $q=\kappa_3$, $r=\kappa_4$, $s=\kappa_5$. Thus, we study the differential equation
\begin{align}\label{eq:WH_homog_ode_pqrs}
\begin{split}
\dot{x} &= x w - r x - p x y, \\
\dot{y} &= s z - q y, \\
\dot{z} &= r x - s z, \\
\dot{w} &= p x y + q y - x w.
\end{split}
\end{align}

\begin{theorem}\label{thm:WH_homog}
For the mass-action system \eqref{eq:WH_homog_ode_pqrs} the following statements hold.
\begin{enumerate}[(i)]
\item In the stoichiometric class $x+y+z+w=c$ there is no positive equilibrium if $c\leq r$ and there is exactly one positive equilibrium if $c>r$. The set of positive equilibria is given by
\begin{align}\label{eq:WH_homog_equilibrium}
(x^*(t),y^*(t),z^*(t),w^*(t)) = t \left(1,\frac{r}{q},\frac{r}{s},\frac{pr}{q}\right) + (0,0,0,r) \text{ for } t>0.
\end{align}
\item The positive equilibrium \eqref{eq:WH_homog_equilibrium} is asymptotically stable on $\mathcal{H}_s$, undergoes an Andronov--Hopf bifurcation at $\mathcal{H}_c$, and is unstable on $\mathcal{H}_u$, where
\begin{align*}
\mathcal{H}_s &= \{(p,q,r,s,t)\in\mathbb{R}^5_+ \colon h(p,q,r,s,t)>0\},\\
\mathcal{H}_c &= \{(p,q,r,s,t)\in\mathbb{R}^5_+ \colon h(p,q,r,s,t)=0\},\\
\mathcal{H}_u &= \{(p,q,r,s,t)\in\mathbb{R}^5_+ \colon h(p,q,r,s,t)<0\}
\end{align*}
with $h(p,q,r,s,t)=(q+r+s)t^2 + [(q+s)^2-prs]t + qs(q+s)$.
\item On $\mathcal{H}_c$, up to a positive factor, the first focal value, $L_1$, equals
\begin{align}\label{eq:WH_homog_L1}
(q+r+s)t^2-(q+s)(q+r+s)t-2qs(q+s).
\end{align}
In particular, the Andronov--Hopf bifurcation could be supercritical, subcritical, or degenerate.
\item For fixed $p,q,r,s>0$ there exists exactly $\begin{cases} 0\\1\\2 \end{cases}$ positive $t$ with
\begin{align*}
(p,q,r,s,t)\in\mathcal{H}_c \text{ if }prs\begin{cases}< \\=\\> \end{cases} (q+s)^2+2\sqrt{qs(q+s)(q+r+s)}.
\end{align*}
\item For $p,q,r,s>0$ with $prs>(q+s)^2+2\sqrt{qs(q+s)(q+r+s)}$, the first focal value is negative at the smaller $t$ with $(p,q,r,s,t)\in\mathcal{H}_c$.
\item There exist rate constants and there exists a stoichiometric class such that the positive equilibrium is asymptotically stable and is surrounded by two limit cycles (an unstable and a stable one).
\item For $c>r$ the system is permanent in the positive stoichiometric class $x+y+z+w=c$.
\end{enumerate}
\end{theorem}
\begin{proof}
Statement (i) follows by a direct calculation.

To prove statement (ii), we eliminate $w$ using the conservation law $x+y+z+w=c$, where $c = x^*(t) + y^*(t) + z^*(t) + w^*(t)$. Thus, we are left with a differential equation in $3$ variables. The $3\times 3$ Jacobian matrix, $J$, at $(x^*(t),y^*(t),z^*(t))$ is given by
\begin{align*}
J = \begin{bmatrix}
-t & -(p+1)t & -t \\
0  & -q      & s \\
r  & 0       & -s
\end{bmatrix}.
\end{align*}
The characteristic polynomial of $J$ is
\begin{align*}
\lambda^3 + (q+s+t)\lambda^2 + (qs+qt+rt+st)\lambda + (qr+qs+rs+prs)t.
\end{align*}
It has a real and negative root. By the Routh--Hurwitz criterion one finds that the other two eigenvalues
\begin{itemize}
\item have negative real part if $(p,q,r,s,t)\in\mathcal{H}_s$,
\item are purely imaginary if $(p,q,r,s,t)\in\mathcal{H}_c$,
\item have positive real part if $(p,q,r,s,t)\in\mathcal{H}_u$.
\end{itemize}
This concludes the proof of (ii).

Application of formula \cite[(5.39)]{kuznetsov:2004} shows that the first focal value indeed has the same sign as \eqref{eq:WH_homog_L1} (for the calculation, see the respective Mathematica Notebook in the GitHub repository \cite{boros:2022}).
With $p=8$, $q=1$, $r=2$, $s=1$, $t=\frac{3\pm\sqrt{7}}{2}$ we have $(p,q,r,s,t)\in \mathcal{H}_c$, and by \eqref{eq:WH_homog_L1} we have $L_1<0$ for $t=\frac{3-\sqrt{7}}{2}$ and $L_1>0$ for $t=\frac{3+\sqrt{7}}{2}$. Additionally, with $p=3(1+\sqrt{2})$, $q=1$, $r=2$, $s=1$, $t=1+\sqrt{2}$ we have $(p,q,r,s,t)\in \mathcal{H}_c$ and $L_1=0$. Thus, the Andronov--Hopf bifurcation could indeed be supercritical, subcritical, or degenerate. This concludes the proof of (iii).

Statement (iv) follows by a short calculation.

To prove (v), we first introduce the parameters
\begin{align*}
a &= q+r+s,\\
b &= prs-(q+s)^2,\\
c &= qs(q+s),\\
d &= qs.
\end{align*}
With this, for fixed $p$, $q$, $r$, $s>0$, the quadratic function $t\mapsto h(p,q,r,s,t)$ has exactly two positive roots provided that $b>0$ and $b^2-4ac>0$. The smaller one of these two roots equals $\frac{b-\sqrt{b^2-4ac}}{2a}$. In terms of the new parameters, formula \eqref{eq:WH_homog_L1} becomes $at^2-\frac{ac}{d}t-2c$, which is negative for $0<t<\frac{\frac{ac}{d}+\sqrt{\frac{a^2c^2}{d^2}+8ac}}{2a}$. Thus, statement (v) follows once we show that
\begin{align} \label{eq:WH_abcd}
b-\sqrt{b^2-4ac} < \frac{ac}{d}+\sqrt{\frac{a^2c^2}{d^2}+8ac}
\end{align}
holds for all $a>0$, $b>0$, $c>0$, $d>0$ with $b^2-4ac>0$. To show the inequality \eqref{eq:WH_abcd}, estimate the r.h.s.\ from below by $\sqrt{4ac}$ and add $\sqrt{b^2-4ac}$ to both sides. Then we arrive at $b < \sqrt{4ac} + \sqrt{b^2-4ac}$, which holds under the stated assumptions on $a$, $b$, $c$, because $\sqrt{A^2+B^2}<A+B$ holds for all $A, B>0$.

To prove (vi), set $q=1$, $r=2$, $s=1$ and keep $p$ and $t$ as parameters. Since $h(p,1,2,1,t) = -2(2t^2+(2-p)t+1)$, an Andronov--Hopf bifurcation happens at the curve $2t^2+(2-p)t+1=0$ in the $(t,p)$-plane. Along this curve, the first focal value vanishes at $(t^*,p^*)=(1+\sqrt{2},3(1+\sqrt{2}))$, while it is negative (respectively, positive) for $t<t^*$ (respectively, $t>t^*$), see the left panel in \Cref{fig:WH}. Using the method described in \cite[Sections 8.7.1 and 8.7.3]{kuznetsov:2004}, one computes the second focal value at $(t^*,p^*)$ and finds it is negative (see the respective Mathematica Notebook in the GitHub repository \cite{boros:2022} for the details). Thus, for $(t^*,p^*)$, the positive equilibrium is asymptotically stable. Perturb slightly $(t,p)$ along the Andronov--Hopf curve in the direction $t>t^*$, with this the first focal value becomes positive, the equilibrium is repelling on the center manifold, and is surrounded by a stable limit cycle that is born via a Bautin bifurcation. Finally, perturb away from the Andronov--Hopf curve in such a way that the real parts of the nonreal eigenvalues at the equilibrium become negative. With this, an unstable limit cycle is created via a subcritical Andronov--Hopf bifurcation.

To prove (vii), we use \Cref{thm:permanence2a}. For any fixed $c>r$ consider the mass-action system \eqref{eq:WH_homog_ode_pqrs}
in the tetrahedron 
\begin{align*}
\Delta=\{(x,y,z,w)\in\mathbb{R}^4_{\geq0}\colon x+y+z+ w = c\}.
\end{align*}
The only equilibrium on the boundary of $\Delta$ is  $E_\mathsf{W} = (0,0,0,c)$. One finds that the eigenvalues of the Jacobian matrix at $E_\mathsf{W}$  within $\Delta$ are $-q$, $-s$, $c-r$, two of them are negative and one is positive for $c>r$ ($E_\mathsf{W}$  undergoes a transcritical bifurcation at $c=r$, giving birth to the unique positive equilibrium \eqref{eq:WH_homog_equilibrium}). Notice that the face $\{(x,y,z,w)\in\Delta \colon x=0\}$ is forward invariant and every solution starting on this face converges to $E_\mathsf{W}$. Thus, this face is the stable manifold of $E_\mathsf{W}$, and we have $W^s(E_\mathsf{W}) \subseteq \partial \Delta$. Since the third eigenvalue at $E_\mathsf{W}$ is positive, no solution starting in the interior of $\Delta$ can converge to $E_\mathsf{W}$. 
By \Cref{thm:permanence2a} with  $M = \{E_\mathsf{W}\}$, the system is then permanent. 
\end{proof}

\begin{figure}[ht]
    \centering
    \begin{tabular}{cc}
    \includegraphics[scale=0.25]{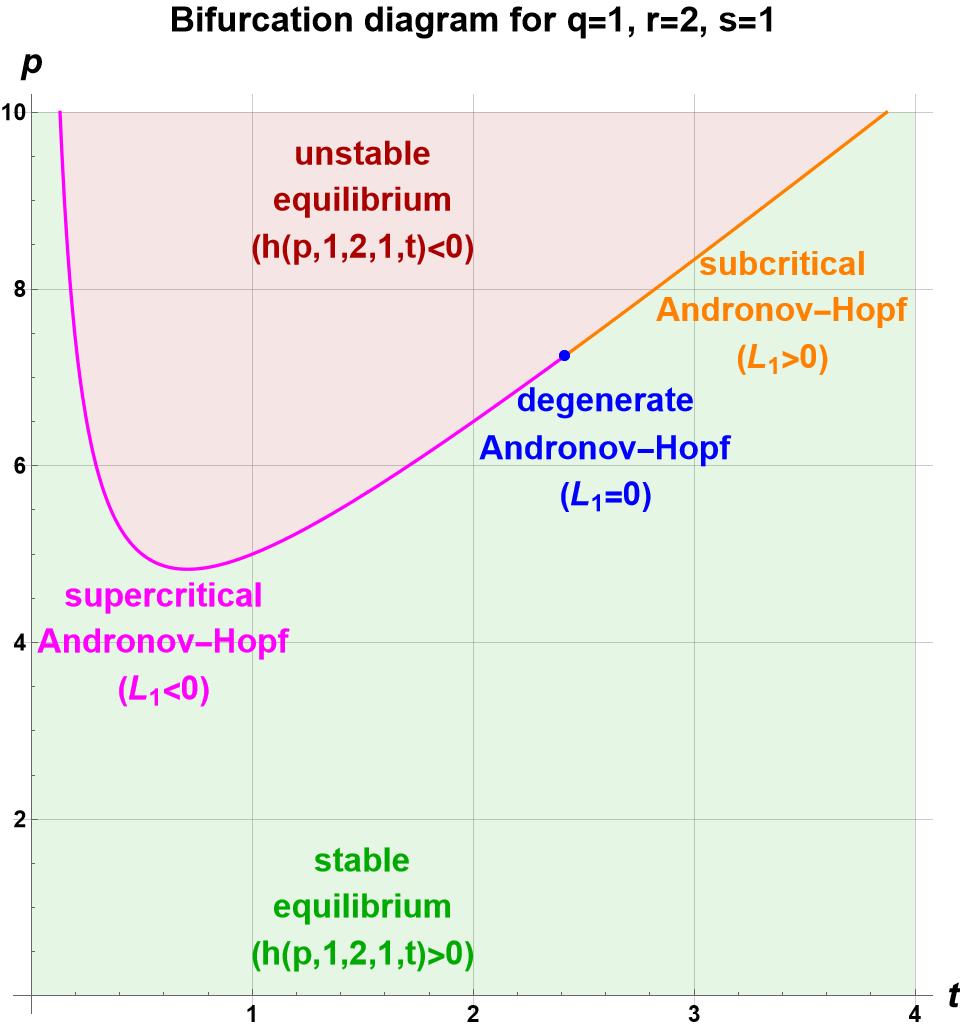} &
    \includegraphics[scale=0.19,trim=1.8cm 0 0 0]{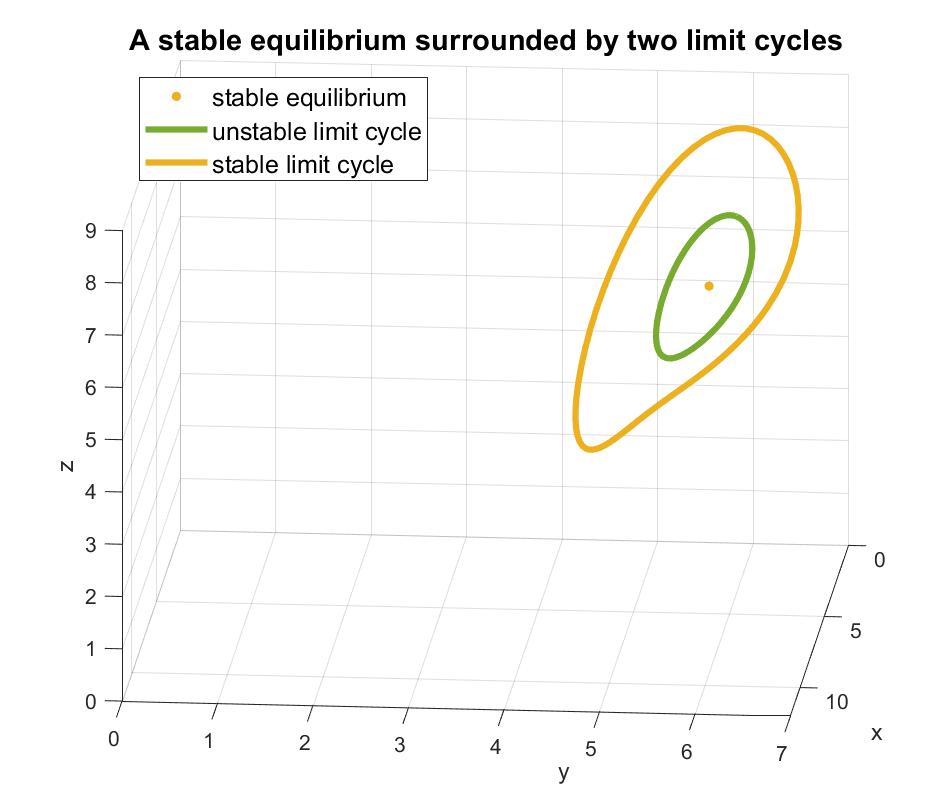}
    \end{tabular}
    \caption{Take the mass-action system \eqref{eq:WH_homog_ode_pqrs} with $q = 1$, $r = 2$, $s = 1$. Left panel: the bifurcation diagram with $p$ and $t$ kept as parameters. Right panel: the two limit cycles for $p=8$ and $t=2.84$ (notice that these parameters fall slightly to the right of the orange curve of the left panel); the figure was created using MATCONT \cite{dhooge:govaerts:kuznetsov:2003}.}
    \label{fig:WH}
\end{figure}

%% file: tikz/WH_original.tex
\begin{tikzpicture}[scale=1.7]

\node (P1) at (0,0)  {$\mathsf{X}$};
\node (P2) at (0,2/3)  {$2\mathsf{X}$};
\node (P3) at (1,0)  {$\mathsf{Z}$};
\node (P4) at (2,0)  {$\mathsf{Y}$};
\node (P5) at (2,2/3) {$\mathsf{X} + \mathsf{Y}$};
\node (P6) at (3,0) {$\mathsf{0}$};

\draw[arrows={-stealth}] (P1) to node[left] {$\kappa_1$} (P2);
\draw[arrows={-stealth}] (P5) to node[left] {$\kappa_2$} (P4);
\draw[arrows={-stealth}] (P4) to node[above] {$\kappa_3$} (P6);
\draw[arrows={-stealth}] (P1) to node[above] {$\kappa_4$} (P3);
\draw[arrows={-stealth}] (P3) to node[above] {$\kappa_5$} (P4);

\node at (4.5,1/3)
{$\begin{aligned}
\dot{x} &= (\kappa_1 - \kappa_4)x - \kappa_2 x y, \\
\dot{y} &= \kappa_5 z - \kappa_3 y, \\
\dot{z} &= \kappa_4 x - \kappa_5 z.
\end{aligned}$};

\end{tikzpicture}

%% file: tikz/WH_partial_homog.tex
\begin{tikzpicture}[scale=1.7]

\node (P1) at (0,0)   {$\mathsf{X}+\mathsf{W}$};
\node (P2) at (1,0)   {$2\mathsf{X}$};
\node (P3) at (0,-2/3)  {$\mathsf{X}+\mathsf{Y}$};
\node (P4) at (1,-2/3)  {$\mathsf{Y}+\mathsf{W}$};
\node (P5) at (-1,-4/3) {$\mathsf{X}$};
\node (P6) at (0,-4/3)  {$\mathsf{Z}$};
\node (P7) at (1,-4/3)  {$\mathsf{Y}$};
\node (P8) at (2,-4/3)  {$\mathsf{W}$};

\draw[arrows={-stealth}] (P1) to node[above] {$\kappa_1$} (P2);
\draw[arrows={-stealth}] (P3) to node[above] {$\kappa_2$} (P4);
\draw[arrows={-stealth}] (P7) to node[above] {$\kappa_3$} (P8);
\draw[arrows={-stealth}] (P5) to node[above] {$\kappa_4$} (P6);
\draw[arrows={-stealth}] (P6) to node[above] {$\kappa_5$} (P7);

\node at (3.8,-2/3)
{$\begin{aligned}
\dot{x} &= \kappa_1 x w - \kappa_4 x - \kappa_2 x y, \\
\dot{y} &= \kappa_5 z - \kappa_3 y, \\
\dot{z} &= \kappa_4 x - \kappa_5 z, \\
\dot{w} &= \kappa_2 x y + \kappa_3 y - \kappa_1 x w.
\end{aligned}$};

\end{tikzpicture}

%% file: sections/5_W.tex
\newpage
\section{Wilhelm oscillator} \label{sec:W}

In 2009, Thomas Wilhelm mentioned a bimolecular reaction network with $3$ species and $4$ reactions that admit a supercritical Andronov--Hopf bifurcation, and thus, a stable limit cycle \cite[Discussion]{wilhelm:2009}. Namely, he introduced the mass-action system
\begin{align}\label{eq:W_original_ode}
\begin{aligned}
\input{tikz/W_original}
\end{aligned}
\end{align}

Since two-species bimolecular mass-action systems do not admit a limit cycle \cite{pota:1983}, one needs at least three species for the existence of a limit cycle. For a three-species bimolecular mass-action system to admit a limit cycle, the underlying network must have a three-dimensional stoichiometric subspace \cite{pota:1985}. A three-species mass-action system with only three reactions and a three-dimensional stoichiometric subspace has deficiency zero, and hence, by \Cref{thm:dfc0_noLC}, does not admit any periodic solution. Consequently, in this sense, the bimolecular mass-action system \eqref{eq:W_original_ode} is a minimal one with a limit cycle. A classification of all bimolecular mass-action systems with three species and four reactions that admit an Andronov--Hopf bifurcation has been recently worked out in \cite{banaji:boros:2022}.

We now analyse the Wilhelm oscillator \eqref{eq:W_original_ode}. In particular, we show that the Andronov--Hopf bifurcation is always supercritical.
\begin{theorem} \label{thm:W_original}
For the mass-action system \eqref{eq:W_original_ode} the following statements hold.
\begin{enumerate}[(i)]
\item The system is competitive.
\item There exists a unique positive equilibrium, it is given by
\begin{align*}
(x^*,y^*,z^*) = \left(\sqrt{\frac{2\kappa_4}{\kappa_2}}\frac{\kappa_1}{\kappa_3},\frac{4\kappa_1\kappa_4}{\kappa_3^2},\frac{\kappa_1}{\kappa_3}\right).
\end{align*}
\item The unique positive equilibrium undergoes a supercritical Andronov--Hopf bifurcation at $\kappa_3 = 4(\kappa_4 + \sqrt{2\kappa_2 \kappa_4})$, the first focal value is negative.
\item The unique positive equilibrium is
\begin{center}
$\begin{cases}
\text{asymptotically stable} \\ \text{unstable}
\end{cases}$ for $\kappa_3 \begin{cases} \leq \\ >\end{cases} 4(\kappa_4 + \sqrt{2\kappa_2 \kappa_4})$.
\end{center}
\item The unique positive equilibrium is surrounded by a stable limit cycle for $\kappa_3$ slightly larger than $4(\kappa_4 + \sqrt{2\kappa_2 \kappa_4})$.
\end{enumerate}
\end{theorem}
\begin{proof}
With $\widetilde{x}=-x$, the differential equation \eqref{eq:W_original_ode} in the new coordinates $(\widetilde{x},y,z)$ takes the form
\begin{align*}
\dot{\widetilde{x}} &= 2\kappa_2 \widetilde{x}^2 - \kappa_3 yz,   \\
\dot{y} &= \kappa_1 y - \kappa_3 yz, \\
\dot{z} &= \kappa_2 \widetilde{x}^2 - 2\kappa_4 z^2
\end{align*}
with state space $\mathbb{R}_- \times \mathbb{R}_+ \times \mathbb{R}_+$. The Jacobian matrix of the r.h.s.\ equals
\begin{align*}
\begin{bmatrix}
4\kappa_2 \widetilde{x}            & -\kappa_3 z           & -\kappa_3 y \\
0                       & \kappa_1 - \kappa_3 z & -\kappa_3 y \\
2\kappa_2 \widetilde{x} & 0                     & -4\kappa_4 z
\end{bmatrix},
\end{align*}
a matrix whose off-diagonal entries are nonpositive on $\mathbb{R}_- \times \mathbb{R}_+ \times \mathbb{R}_+$, the system is thus competitive, proving (i).

Statement (ii) follows by a direct calculation.

To prove (iii), (iv), and (v), we multiply the r.h.s.\ of the differential equation \eqref{eq:W_original_ode} by the positive constant $\frac{\kappa_3}{\kappa_1\kappa_2}$ and introduce the new parameters $p=\frac{\kappa_3}{\kappa_2}$, $q = \sqrt{\frac{2\kappa_4}{\kappa_2}}$, and $r=\frac{\kappa_3}{\kappa_1}$. With this, the differential equation to be studied is
\begin{align*}
\dot{x} &= pr yz - 2r x^2,   \\
\dot{y} &= p y - pr yz, \\
\dot{z} &= r x^2 - q^2 r z^2
\end{align*}
with the unique positive equilibrium being $(x^*,y^*,z^*) = \left(\frac{q}{r}, \frac{2q^2}{pr}, \frac{1}{r}\right)$. The Jacobian matrix, $J$, at $(x^*,y^*,z^*)$ is given by
\begin{align*}
J = \begin{bmatrix}
-4q & p &  2q^2 \\
  0 & 0 & -2q^2 \\
 2q & 0 & -2q^2
\end{bmatrix}.
\end{align*}
The characteristic polynomial of $J$ equals $\lambda^3 + 2q(q+2)\lambda^2 + 4q^3 \lambda + 4pq^3$. By the Routh--Hurwitz criterion, all the three eigenvalues of $J$ have negative real part if and only if $p < 2q(q+2)$. For $p = 2q(q+2)$, the eigenvalues are $\{-p,\omega i,-\omega i\}$, where $\omega = \sqrt{\frac{\det J}{\tr J}} = 2q^\frac32$. When $p > 2q(q+2)$, one eigenvalue is real and negative, the other two eigenvalues have positive  real part. Since $p \lesseqgtr 2q(q+2)$ if and only if $\kappa_3 \lesseqgtr 4(\kappa_4 + \sqrt{2\kappa_2 \kappa_4})$, the statements (iv) and (v) will follow once we prove that the first focal value is negative. Indeed, application of formula \cite[(5.39)]{kuznetsov:2004} shows that the first focal value is negative. For the calculation, see the respective Mathematica Notebook in the GitHub repository \cite{boros:2022}.
\end{proof}


One way to construct a mass-conserving bimolecular mass-action system with a limit cycle is to homogenise the Wilhelm oscillator. This leads to the mass-action system
\begin{align}\label{eq:W_homog_ode}
\begin{aligned}
\input{tikz/W_fully_homog}
\end{aligned}
\end{align}
Since a limit cycle that is born via an Andronov--Hopf bifurcation is hyperbolic (at least when the parameters are close enough to the bifurcation point in parameter space), the homogenised Wilhelm oscillator \eqref{eq:W_homog_ode} does admit a stable limit cycle by \cite[Theorem 1]{banaji:boros:hofbauer:2022}.

At the beginning of this section we explained in what sense the Wilhelm oscillator \eqref{eq:W_original_ode} is minimal. We now argue that the homogenised Wilhelm oscillator \eqref{eq:W_homog_ode} is minimal among the mass-conserving bimolecular mass-action systems with a limit cycle. Since three-species bimolecular mass-action systems with a two-dimensional stoichiometric subspace do not admit a limit cycle \cite{pota:1985}, one needs at least four species. Since no bimolecular mass-action system with a two-dimensional stoichiometric subspace admits a limit cycle, the stoichiometric subspace has to be three-dimensional at least. A four-species mass-action system with only three reactions and a three-dimensional stoichiometric subspace is of deficiency zero, and hence, does not admit any periodic solution, see \Cref{thm:dfc0_noLC}. Consequently, in this sense, the mass-conserving bimolecular mass-action system \eqref{eq:W_homog_ode} is a minimal one with a limit cycle.

We now analyse the homogenised Wilhelm oscillator \eqref{eq:W_homog_ode}. In particular, we show that the Andronov--Hopf bifurcation is always supercritical.
\begin{theorem} \label{thm:W_homog}
For the mass-action system \eqref{eq:W_homog_ode} the following statements hold.
\begin{enumerate}[(i)]
\item The phase portrait in the stoichiometric classes $x+y+z+w=c_1$ and $x+y+z+w=c_2$ are identical (up to a scaling of all the variables by $\frac{c_2}{c_1}$).
\item In every positive stoichiometric class there exists a unique positive equilibrium, the set of positive equilibria is given by
\begin{align*}
(x^*(t),y^*(t),z^*(t),w^*(t)) = t(q,2q^2/p,1,r) \text{ for }t>0,
\end{align*}
where $p=\frac{\kappa_3}{\kappa_2}$, $q = \sqrt{\frac{2\kappa_4}{\kappa_2}}$, and $r=\frac{\kappa_3}{\kappa_1}$.
\item The unique positive equilibrium undergoes a supercritical Andronov--Hopf bifurcation at the surface
\begin{align*}
\mathcal{H}_c=\left\{(p,q,r) \in \mathbb{R}^3_+\colon r>1, q<r(r-1), p = 2q(q+2)\frac{r^2+(q+2)r+q}{r(r-1)-q}\right\}.
\end{align*}
\item The unique positive equilibrium is asymptotically stable (respectively, unstable) for $(p,q,r) \in \mathcal{H}_s \cup \mathcal{H}_c$ (respectively, $(p,q,r) \in \mathcal{H}_u$), where
\begin{align*}
\mathcal{H}_u&=\left\{(p,q,r) \in \mathbb{R}^3_+\colon r>1, q<r(r-1), p > 2q(q+2)\frac{r^2+(q+2)r+q}{r(r-1)-q}\right\},\\
\mathcal{H}_s&=\mathbb{R}^3_+ \setminus (\mathcal{H}_c \cup \mathcal{H}_u).
\end{align*}
\item There exists a stable limit cycle for parameters in $\mathcal{H}_u$ that are close enough to $\mathcal{H}_c$.
\item The system is permanent.
\end{enumerate}
\end{theorem}
\begin{proof}
Statement (i) is an immediate consequence of the fact that the r.h.s.\ is a homogeneous polynomial (of degree two).

Statement (ii) follows by a direct calculation.

To prove (iii), (iv), and (v), we multiply the r.h.s.\ of the differential equation \eqref{eq:W_homog_ode} by the positive constant $\frac{\kappa_3}{\kappa_1\kappa_2}$ and write the differential equation using the parameters $p$, $q$, $r$:
\begin{align*}
\dot{x} &= pr yz - 2r x^2,   \\
\dot{y} &= p yw - pr yz, \\
\dot{z} &= r x^2 - q^2 r z^2, \\
\dot{w} &= r x^2 + q^2 r z^2 - p yw.
\end{align*}
We pick the particular stoichiometric class that contains the equilibrium $(x^*(t),y^*(t),z^*(t),w^*(t))$ with $t = 1$ (by (i), this step does not restrict generality). Next, we eliminate $w$ using the conservation law $x+y+z+w=c$, where $c = x^*(1) + y^*(1) + z^*(1) + w^*(1)$. Thus, we are left with a differential equation in $3$ variables. The $3\times 3$ Jacobian matrix, $J$, at $(x^*(1),y^*(1),z^*(1))$ is given by
\begin{align*}
J = \begin{bmatrix}
-4qr  &    pr &  2q^2r \\
-2q^2 & -2q^2 & -2q^2(r+1) \\
 2qr  &     0 & -2q^2r
\end{bmatrix}.
\end{align*}
The characteristic polynomial of $J$ equals $\lambda^3 + a_2 \lambda^2 + a_1 \lambda + a_0$ with
\begin{align*}
a_2 &= 2 q (q + 2 r + q r), \\
a_1 &= 2 q^2 r (p + 2 q (2 + q + r)), \\
a_0 &= 4 q^3 r^2 (2 q^2 + p (1 + q + r)).
\end{align*}
One confirms by the Routh--Hurwitz criterion that all the three eigenvalues of $J$ have negative real part if and only if $(p,q,r)\in \mathcal{H}_s$. The eigenvalues are $\{\omega i, -\omega i, \varrho\}$ with $\varrho=\tr J$ and $\omega = \sqrt{\frac{\det J}{\tr J}}$ if and only if $(p,q,r)\in\mathcal{H}_c$. Finally, one eigenvalue is real and negative, while the other two eigenvalues have positive real part if and only if $(p,q,r)\in \mathcal{H}_u$. Once we prove that the first focal value is negative, all of the statements (iii), (iv), and (v) follow. Indeed, application of formula \cite[(5.39)]{kuznetsov:2004} shows that the first focal value is negative. For the calculation, see the respective Mathematica Notebook in the GitHub repository \cite{boros:2022}.


To prove (vi), we use \Cref{thm:permanence2b}. For any fixed $c>0$ consider the mass-action system \eqref{eq:W_homog_ode} in the tetrahedron 
\begin{align*}
\Delta=\{(x,y,z,w)\in\mathbb{R}^4_{\geq0}\colon x+y+z+w =c\}.
\end{align*}
The only equilibria on the boundary of $\Delta$ are $E_\mathsf{Y} = (0,c,0,0)$ and $E_\mathsf{W} = (0,0,0,c)$. Let $M_1 = \{E_\mathsf{W}\}$ and $M_2 = \{E_\mathsf{Y}\}$. The maximal invariant subset of $\partial\Delta$ is $M = \{(x,y,z,w) \in \Delta\colon x = z = 0\}$, the edge connecting $E_\mathsf{Y}$ and $E_\mathsf{W}$. The flow there goes from $E_\mathsf{W}$ to $E_\mathsf{Y}$. At $E_\mathsf{Y}$, the eigenvalues within $\Delta$ are $-c\kappa_1$, $0$, $0$. There is a one-dimensional stable manifold (the edge $M \setminus \{E_\mathsf{W}\}$) and a two-dimensional center manifold. Since $\dot{x} + 2 \dot{z} = z(\kappa_3 y - 4\kappa_4 z)$, which is positive in a neighborhood of $E_\mathsf{Y}$, the flow goes away from $E_\mathsf{Y}$ on this center manifold. At $E_\mathsf{W}$, the eigenvalues within $\Delta$ are $c\kappa_1$, $0$, $0$. There is a one-dimensional unstable manifold (the edge $M \setminus \{E_\mathsf{Y}\}$) and a two-dimensional center manifold (the face $\{(x,y,z,w)\in\Delta\colon y=0\}$) which is the stable set $W^s(E_\mathsf{W}) \subseteq \partial \Delta$. Thus, \Cref{thm:permanence2b} applies and shows permanence.
\end{proof}

We conclude this section by collecting a few open questions:
\begin{enumerate}[(i)]
\item Is the mass-action system \eqref{eq:W_original_ode} permanent? The difficulty is that the $y$-axis is invariant and the flow there goes to infinity.
\item For the mass-action systems \eqref{eq:W_original_ode} and \eqref{eq:W_homog_ode}, does local stability of the positive equilibrium imply its global stability?
\item For the mass-action systems \eqref{eq:W_original_ode} and \eqref{eq:W_homog_ode}, whenever the equilibrium is unstable, does there exist a unique limit cycle?
\end{enumerate}

%% file: tikz/W_original.tex
\begin{tikzpicture}[scale=1.7]

\node (P1) at (0,0)    {$\mathsf{Y}$};
\node (P2) at (1,0)    {$2\mathsf{Y}$};
\node (P3) at (0,-1/2) {$2\mathsf{X}$};
\node (P4) at (1,-1/2) {$\mathsf{Z}$};
\node (P5) at (0,-1)   {$\mathsf{Y} + \mathsf{Z}$};
\node (P6) at (1,-1)   {$\mathsf{X} + \mathsf{Z}$};
\node (P7) at (0,-3/2)   {$2\mathsf{Z}$};
\node (P8) at (1,-3/2)   {$\mathsf{0}$};

\draw[arrows={-stealth}] (P1) to node[above] {$\kappa_1$} (P2);
\draw[arrows={-stealth}] (P3) to node[above] {$\kappa_2$} (P4);
\draw[arrows={-stealth}] (P5) to node[above] {$\kappa_3$} (P6);
\draw[arrows={-stealth}] (P7) to node[above] {$\kappa_4$} (P8);

\node at (3,-3/4)
{$\begin{aligned}
\dot{x} &= \kappa_3 yz - 2\kappa_2 x^2,   \\
\dot{y} &= \kappa_1 y - \kappa_3 yz, \\
\dot{z} &= \kappa_2 x^2 - 2\kappa_4 z^2.
\end{aligned}$};

\end{tikzpicture}

%% file: tikz/W_fully_homog.tex
\begin{tikzpicture}[scale=1.7]

\node (P1) at (0,0)    {$\mathsf{Y}+\mathsf{W}$};
\node (P2) at (1,0)    {$2\mathsf{Y}$};
\node (P3) at (0,-1/2) {$2\mathsf{X}$};
\node (P4) at (1,-1/2) {$\mathsf{Z}+\mathsf{W}$};
\node (P5) at (0,-1)   {$\mathsf{Y} + \mathsf{Z}$};
\node (P6) at (1,-1)   {$\mathsf{X} + \mathsf{Z}$};
\node (P7) at (0,-3/2)   {$2\mathsf{Z}$};
\node (P8) at (1,-3/2)   {$2\mathsf{W}$};

\draw[arrows={-stealth}] (P1) to node[above] {$\kappa_1$} (P2);
\draw[arrows={-stealth}] (P3) to node[above] {$\kappa_2$} (P4);
\draw[arrows={-stealth}] (P5) to node[above] {$\kappa_3$} (P6);
\draw[arrows={-stealth}] (P7) to node[above] {$\kappa_4$} (P8);

\node at (3,-3/4)
{$\begin{aligned}
\dot{x} &= \kappa_3 yz -2\kappa_2 x^2,\\
\dot{y} &= \kappa_1 yw - \kappa_3 yz, \\
\dot{z} &= \kappa_2 x^2 - 2\kappa_4 z^2, \\
\dot{w} &= \kappa_2 x^2 + 2\kappa_4 z^2 - \kappa_1 yw.
\end{aligned}$};

\end{tikzpicture}